\documentclass[a4paper,11pt,reqno]{amsart}
\usepackage[british]{babel}
\usepackage{color}
\usepackage[utf8]{inputenc}
\usepackage[paperwidth=210mm]{geometry}
\usepackage{amsfonts,amssymb} % For bbb and frak and circle arrow
\usepackage{mathrsfs}
\usepackage{stmaryrd}

\theoremstyle{plain}
\newtheorem{theorem}{Theorem}[section] 
\newtheorem{lemma}[theorem]{Lemma}
\newtheorem{lemmadef}[theorem]{Lemma/Definition}
\newtheorem{proposition}[theorem]{Proposition}

\theoremstyle{definition}
\newtheorem{definition}[theorem]{Definition}
\theoremstyle{remark}
\newtheorem{remark}[theorem]{Remark}

\numberwithin{equation}{section}

\makeatletter
\def\operatorname#1{\mathop{\operator@font #1}\nolimits}%
\makeatother
\newcommand{\suchthat}{\mathop{\,\vert\,}}
\newcommand{\half}{\tfrac12} % Text style half
\newcommand{\bbC}{\mathbb{C}}
\newcommand{\bbR}{\mathbb{R}}
\newcommand{\CC}{\mathcal{R}}
\newcommand{\EE}{\mathcal{E}}
\newcommand{\calV}{\mathscr{V}}
\newcommand{\WW}{\mathcal{W}}
\newcommand{\calH}{\mathcal{H}}
\newcommand{\scrV}{\mathscr{V}}
\newcommand{\scrW}{\mathscr{W}}
\newcommand{\frako}{\mathfrak{o}}

\newcommand{\fraksp}{\mathfrak{sp}}
\newcommand{\frakgl}{\mathfrak{gl}}
\newcommand{\frakg}{\mathfrak{g}}
\newcommand{\ad}{\operatorname{ad}}

\newcommand{\Tr}{\operatorname{Tr}}

\newcommand{\End}{\operatorname{End}}

\newcommand{\id}{\operatorname{Id}}

\newcommand{\Ker}{\operatorname{Ker}}

\newcommand{\Image}{\operatorname{Image}}
\newcommand{\Id}{\operatorname{Id}}
\newcommand{\ric}{\operatorname{Ric}}

\newcommand*{\cyclic}{\mathop{\kern0.9ex{{+}\kern-2.2ex\raise-.29ex%
  \hbox{\Large\hbox{$\circlearrowright$}}}}\limits}
\newcommand{\eqnref}[1]{(\ref{#1})}
\let\ul\underline
\let\wt\widetilde
\newcommand{\longto}{\longrightarrow}
\let\Wedge\Lambda

\begin{document}
%%
%% Top matter
%%
\title{On Twistor Almost Complex Structures}

\author{Michel Cahen}
\address{D\'epartement de Math\'ematique, Universit\'e Libre de Bruxelles,
Campus Plaine, CP 218, Boulevard du Triomphe, BE-1050 Bruxelles, Belgium.}
\email{mcahen@ulb.ac.be}

\author{Simone Gutt}
\address{D\'epartement de Math\'ematique, Universit\'e Libre de Bruxelles,
Campus Plaine, CP 218, Boulevard du Triomphe, BE-1050 Bruxelles, Belgium.\hfill~\linebreak
\indent Universit\'e de Lorraine,
Institut Elie Cartan de Lorraine, UMR 7502,
Ile du Saulcy, F-57045 Metz, France.}
\email{sgutt@ulb.ac.be}
\thanks{M.~Cahen et S.~Gutt sont Membres de l'Acad\'emie Royale de Belgique\hfill\ \linebreak
\hbox{~~~~To appear in the Journal of Geometric Mechanics}}

\author{John Rawnsley}
\address{Mathematics Institute, Zeeman Building,
University of Warwick, Coventry CV4 7AL, United Kingdom}
\email{j.rawnsley@warwick.ac.uk}

\dedicatory{Dedicated to our friend Kirill Mackenzie}



\begin{abstract}
In this paper we look at the question of integrability, or not, of the two
natural almost complex structures $J^{\pm}_\nabla$ defined on the twistor space
$J(M,g)$ of an even-dimensional manifold $M$ with additional structures $g$ and
$\nabla$ a $g$-connection. We also look at the question of the compatibility of
$J^{\pm}_\nabla$ with a natural closed $2$-form $\omega^{J(M,g,\nabla)}$ defined
on $J(M,g)$. For $(M,g)$ we consider either a pseudo-Riemannian manifold,
orientable or not, with the Levi Civita connection or a symplectic manifold with
a given symplectic connection $\nabla$. In all cases $J(M,g)$ is a bundle of
complex structures on the tangent spaces of $M$ compatible with $g$ and we
denote by $\pi \colon J(M,g) \longrightarrow M$ the bundle projection. In the
case $M$ is oriented we require the orientation of the complex structures to be
the given one. In the symplectic case the complex structures are positive. 

The linear connection $\nabla$ on $M$ defines a horizontal space
$\calH^\nabla_j\simeq T_{\pi(j)}M$ at any point $j$ in the twistor space  so
that $T_jJ(M,g)$ is isomorphic to $\calH^\nabla_j\oplus \scrV_j$ where $\scrV_j
= \Ker \pi_{*j} $ is the vertical space at $j$. Since both $\scrV_j$ and
$TM_{\pi(j)}$ carry complex structures defined by $j$, they add together to give
the complex structure denoted by $(J^+_\nabla)_j$ on $T_jJ(M,g)$. The almost
complex structure denoted $(J^{-}_\nabla)_j$ is defined by reversing the sign on
the horizontal space.

We examine the integrability, or not, of the  $J^{\pm}_\nabla$ by looking at
their Nijenhuis tensors $N^{J^\pm_\nabla}$ and measure their non-integrability
by the dimension of the span of the values of $N^{J^\pm_\nabla}$.

The natural closed $2$-form $\omega^{J(M,g,\nabla)}$ is defined on the twistor
space as the trace of the curvature of a connection $D^E$ defined on the
pull-back bundle bundle $E=\pi^{-1}TM$. This bundle $E$ is endowed with  the
complex vector bundle structure defined by the natural section $\Phi$ of
$\End(E)$ whose value at $j$ is $j$,    and the connection $D^E$, built from
the pullback connection $\pi^{-1}\nabla^E$, satisfies $D^{\End E}\Phi=0$. We
recall, as in Reznikov \cite{bib:Reznikov}, when this $2$-form is symplectic in
the pseudo-Riemannian setting  and we determine, in the pseudo-Riemannian and in
the symplectic setting,  when $\omega^{J(M,g,\nabla)}$ is of type $(1,1)$ with
respect to $J^\pm_\nabla$. 
\end{abstract}

\maketitle
\thispagestyle{empty}

\newpage
   
\tableofcontents

\section*{Introduction}

A twistor space over a manifold $M$ is a fibre bundle $\pi \colon Z \longto M$
where each fibre is a complex manifold and each point $z$ in $Z$ defines a
complex structure $J(z)$ on the tangent space $TM_{\pi(z)}$ ($M$ must be
even-dimensional for this to be possible). An example is the bundle $J(M)$ of
all complex structures $j$ on the tangent spaces of $M$. The case of interest
here is  the bundle $\pi \colon J(M,g) \longto M$  of complex structures on the
tangent spaces compatible with some geometric structure $g$ such as a
pseudo-Riemannian metric (with an orientation or not) or a symplectic structure.
Where we can we will treat those results common to the three cases together. The
presentation we give  of twistor spaces follows the Riemannian case in
O'Brian--Rawnsley \cite{bib:OB-R}.

A linear connection $\nabla$ on $M$ preserving $g$ defines a horizontal space
$\calH^\nabla_j$ at $j$ so that $T_jJ(M,g)$ is isomorphic to $T_{\pi(j)}M\oplus
\scrV_j$ where $\scrV_j = \Ker \pi_* \colon T_jJ(M,g) \longto T_{\pi(j)}M$ is
the vertical space at $j$. Since both $\scrV_j$ and $TM_{\pi(j)}$ carry complex
structures defined by $j$ (which we recall in Section \ref{section:Jpm}), they
add together to give a complex structure $(J^+_\nabla)_j$ on $T_jJ(M,g)$. This
almost complex structure $J^{+}_\nabla$ on $J(M,g)$ can sometimes be integrable
producing a complex manifold which has been  used in the pseudo-Riemannian
setting to convert the Yang--Mills equations on $M$ into the Cauchy--Riemann
equations on $J(M,g)$ in the $4$-dimensional case, see \cite{bib:AHS}. Some
twistor spaces over Riemannian manifolds have been a source of examples of
non-K\"ahlerian symplectic manifolds \cite{bib:Fine-Panov,bib:Fine-Panov2,bib:Reznikov}.

A second almost complex structure $J^{-}_\nabla$ can be defined by reversing the
sign on the horizontal bundle. This has had many uses in the study of harmonic
maps of Riemann surfaces into $M$ when $M$ has a Riemannian structure $g$ and
$\nabla$ is the Levi Civita connection of $g$ \cite{bib:EellsSalamon}.

In this paper we  look at the question of integrability, or not, of
$J^{\pm}_\nabla$ and, when not integrable, examine their Nijenhuis tensors
$N^{J^\pm_\nabla}$ to see how non-integrable they are, using as a measure of
their non-integrability the dimension of the span of the values of
$N^{J^\pm_\nabla}$, as  in \cite{bib:CGHG}.

The bundle $\End(E)$, where $E$ is the pull-back bundle $E=\pi^{-1}TM$ (which is
isomorphic to the horizontal bundle $\calH^\nabla$ via $\pi_*$), has a section
$\Phi$ whose value at $j$ is $j$. This makes $(E,\Phi)$ into a complex vector
bundle with the multiplication by $\sqrt{-1}$ given by $\Phi$.  This complex
vector bundle has Chern classes $c_i(E,\Phi)$ in the de Rham cohomology of
$J(M,g)$ represented by polynomials in the curvature of a connection on $E$
preserving $\Phi$. From the pullback connection $\pi^{-1}\nabla$, we get such a
connection on $E$, called $D^E$, and construct a closed 2-form
$\omega^{J(M,g,\nabla)}$ as the trace of the curvature of $D^E$. We write the
conditions for this $2$-form to be symplectic  and we determine when
$\omega^{J(M,g,\nabla)}$ is of type $(1,1)$ with respect to $J^\pm_\nabla$.

\medskip
\noindent The results in the pseudo-Riemannian context include the following:

The almost complex structure  $J^+_\nabla$ is integrable in the
pseudo-Riemannian context with no given orientation if and only if the Weyl
component $C^\nabla$ of the Riemann curvature $R^\nabla$ vanishes (this is well
known and proven in 

Proposition \ref{prop:pseudoRnonor}).

In the pseudo-Riemannian context with a  given orientation, the results holds
true (as is well known) in dimension $>4$: $J^+_\nabla$ is integrable  if and
only if the Weyl component of the Riemann curvature vanishes, whether in
dimension $4$ it is integrable if and only if the Weyl component of the Riemann
curvature tensor is self-dual when the signature is $(4,0)$ or $(0,4)$ and
anti-self-dual when the signature is $(2,2)$ (Proposition \ref{prop:pseudoRor}).

The almost complex structure $J^-_\nabla$ is never integrable and the image of
its Nijenhuis tensor always include the horizontal space: $\Image
N^{J^-_\nabla}_j\supset \calH^\nabla_j$.

If the space has non-vanishing constant sectional curvature, then the image of
the Nijenhuis tensor associated to $J_\nabla^-$ is the whole tangent space
$T_jJ(M,g)$ at any point $j\in J(M,g)$.

More generally in the Riemannian case (Proposition \ref{prop:pinchedcurv}),
given any positive integer $n$, there exists an $\epsilon(n)$ such that, if the
sectional curvature of a Riemannian manifold $(M,g)$ of dimension $2n$ is
$\epsilon(n)$-pinched, the almost complex structure $J^-_\nabla$ on the twistor
space, defined using the Levi Civita connection $\nabla$, is maximally
non-integrable (i.e. the image of the corresponding Nijenhuis tensor is the
whole tangent space at every point).

Each of the complex structures $J^\pm_\nabla$ is compatible with the closed
2-form $\omega^{J(M,g,\nabla)}$ if and only if the same condition as the
integrability of $J^+_\nabla$ is satisfied, i.e. $\omega^{J(M,g,\nabla)}$ is of
type $(1,1)$ with respect to $J^+_\nabla$ (and automatically also to
$J^-_\nabla$) if and only if $C^\nabla=0$ in the pseudo-Riemannian context with
no orientation, or with an orientation if $\dim M>4$ and if and only if the Weyl
component of the Riemann curvature tensor is self-dual when the signature is
$(4,0)$ or $(0,4)$ and anti-self-dual when the signature is $(2,2)$ (Proposition
\ref{prop:Appendix}).

\medskip
\noindent The results in the symplectic context include the following:

The almost complex structure $J^+_\nabla$ on the twistor space $J(M,\omega)$ of
a symplectic manifold $(M,\omega)$ of dimension $2n\ge 4$, defined using a
symplectic connection $\nabla$, is integrable if and only if the curvature of
$\nabla$ is of Ricci-type (this was known and is proven in Proposition
\ref{prop:sympl1}).

The almost complex structure $J^-_\nabla$ is never integrable and the image of
its Nijenhuis tensor at the point $j$ always include the horizontal space
$\calH^\nabla_j$.

 The closed $2$-form $\omega^{J(M,\omega,\nabla)}$ is of type $(1,1)$ for each
 of the $J^\pm_\nabla$ if and only if  again the same condition as the
 integrability of $J^+_\nabla$ is satisfied, i.e. the curvature $R^\nabla$ is of
 Ricci type (Proposition \ref{prop:sympl}).

\section*{Acknowledgement}
JR has the pleasure of thanking MC and SG for their hospitality in Brittany
where part of this work was done. This work benefited from the project
``Symplectic techniques in differential geometry", funded by the ``Excellence of
Science (EoS)" program 2018--2021 of the FWO/F.R.S-FRNS. 

\section{Description of the twistor bundle}\label{section:twistor}

Let $(M,g)$ be a $2n$-dimensional manifold endowed with a structure $g$ which
can be either a (pseudo)-Riemannian structure of signature $(2p,2q)$ where
$n=p+q$, with an orientation or not, or a symplectic structure, or having no
extra structure. 

Let $F(M,g)\rightarrow M$ denote the corresponding frame bundle
where a frame at a point $p\in M$ is  a map $\xi: V \rightarrow T_pM$ which is a
linear  isomorphism from $V=\bbR^{2n}$, endowed with a standard structure
$\wt{g}_0$,  to $(T_pM,g_p)$, where
$\wt{g}_0=\begin{pmatrix}I_{p,q}&0\\0&I_{p,q}\end{pmatrix}$ with
$I_{p,q}=\begin{pmatrix}I_{p}&0\\0&-I_{q}\end{pmatrix}$ in the pseudo-Riemannian
case, with an orientation or not, and
$\wt{g}_0=\Omega_0=\begin{pmatrix}0&I_n\\-I_n&0\end{pmatrix}$ in the symplectic
case. 

The frame bundle is a principal bundle with structure group
\[
G=Gl(V,\wt{g}_0)=\left\{\begin{array}{l}
O(V,\wt{g}_0)\simeq O(2p,2q;\bbR) \textrm{ in the pseudo-Riemannian setting;}\\
SO(V,\wt{g}_0)\textrm{ when there is furthermore an orientation;}\\
Sp(V,\Omega_0) \textrm{ in the symplectic case;}\\
Gl(V)=Gl(2n,\bbR) \textrm{ if there is no extra structure on }M.\\
\end{array}\right.
\]

The twistor bundle,  $J(M,g)\stackrel{\pi}{\rightarrow}M$, is the bundle whose
fibre over a point $p$ of $M$ consists of all complex structures $j$ on $T_pM$
which are compatible with $g_p$  in the sense  that there is a frame  at the point
$p$, $\xi$ in the fibre  $F(M,g)_p$, in which the complex structure can be
written $j=\xi\circ \wt{j_0}\circ \xi^{-1}$ where
$\wt{j_0}:=\begin{pmatrix} 0&-\id_n\\ \id_n & 0\end{pmatrix}$ (so
we mean in particular positive compatible almost complex structures in the
symplectic case, and we mean that $j_0$ is compatible with the orientation when
an orientation is given in the pseudo-Riemannian case).

Observe that a complex structure $\wt{j}$ on $V$  is compatible with
$\wt{g}_0$ if  there exists a  basis of $V$, compatible with
$\wt{g}_0$, in which the matrix associated to $\wt{j}$ is
$\wt{j_0}$, hence $\wt{j}=A\wt{j_0} A^{-1}$  with $A\in
G=Gl(V,\wt{g}_0)$ and the space of such complex structures identifies
with $Gl(V,\wt{g}_0)/ Gl(V,\wt{g}_0,\wt{j_0})$ with
\[
Gl(V,\wt{g}_0,\wt{j_0})=\left\{A\in G \mid A \wt{j_0} 
= \wt{j_0} A \right\}\simeq
\left\{\begin{array}{l}U(p,q)  \textrm{ in the pseudo-Riemannian setting;}\\
U(n) \textrm{ in the symplectic case;}\\
Gl(n,\bbC)\textrm{ if there is no extra structure.}\end{array}\right.
\]

The twistor bundle  $J(M,g)$ can thus be seen as a quotient of the frame bundle:
\begin{equation}
J(M,g)= F(M,g)\times _G\left(G/ Gl(V,\wt{g}_0,\wt{j_0})\right)
=F(M,g)/Gl(V,\wt{g}_0,\wt{j_0})
\end{equation}
and we shall denote by $\pi_1$ the natural projection (giving a 
$Gl(V,\wt{g}_0,\wt{j_0})$-principal bundle structure):
\begin{equation}\label{eq:pi1}
\pi_1: F(M,g)\rightarrow J(M,g)
=F(M,g)/Gl(V,\wt{g}_0,\wt{j_0}): 
\xi \mapsto j=\xi \circ \wt{j_0}\circ \xi^{-1}.
\end{equation}

%%%%%%%%%%%%%%%%%%%%%%%%%%%%%%%

\section{Almost complex structures on  the twistor space}\label{section:Jpm}
 
We shall denote by $\calV$ the vertical tangent bundle to the twistor space
\[
\calV_j:=\Ker \pi_{*j}.
\]
Note that a vector in $T_jJ(M,g)$ is vertical if and only if it is tangent to
the fibre, i.e. tangent to a curve $j_t$ of compatible complex structures on
$T_pM$, with $p=\pi(j)$ and $j_0=j$; hence:
\begin{eqnarray}
\calV_j&=&\{ S\in \End (T_pM)\, \vert \, Sj+jS=0,\, g_p(SX,Y)+g_p(X,SY)=0,
\, \forall X,Y \in T_pM\, \}\nonumber \\
&=&\{ [j,S^\prime ]  \,\vert\, S^\prime \in \End (T_pM) \textrm{ and }\, g_p(S^\prime X,Y)
+g_p(X,S^\prime Y)=0
\,\} \label{eq:[j,s]}.
\end{eqnarray}
(indeed, given $S$ in the first set, one can define $S^\prime =\half Sj$ in the 
second set). Let us denote by  $\End(TM,g)$ the bundle of infinitesimal isometries of
the tangent bundle:
\begin{equation}
\End(TM,g)_p:=\{ S\in \End (T_pM)\, 
\vert \, g_p(SX,Y)+g_p(X,SY)=0, \, \forall X,Y \in T_pM\, \}
\end{equation}
and consider the pullback bundles over $J(M,g)$:
\begin{eqnarray}
E:&=&\pi^{-1}TM=\{ (j,X)\in J(M,g)\times TM \, 
\vert \, X\in T_pM \, \textrm{ with } p=\pi(j)\, \}\\
\End (E,g):&=&\pi^{-1} \End(TM,g)=\{ (j,S), \,  j\in J(M,g)_p, \, 
S\in \End(TM,g)_p,\, p\in M \}\nonumber .
\end{eqnarray}
Clearly $\calV$ is a subbundle of $\End (E,g)$. The  canonical section
\begin{equation}\label{eq:defPhi}
\Phi: J(M,g)\rightarrow  \End (E,g): j\mapsto \Phi(j):=(j,j)
\end{equation}
defines the canonical (tautological) complex structure in the bundle $E$.
Using equation \eqnref{eq:[j,s]}, we can write
\[
\calV= [ \Phi,  \End (E,g)].
\]
We  have a short exact sequence of bundles over $J(M,g)$:
\[
0\rightarrow \calV\hookrightarrow 
TJ(M,g)\stackrel{\pi_*}{\rightarrow} E\rightarrow 0.
\]
The datum of a linear connection $\nabla$ on $M$ which preserves the structure
$g$ (i.e. $\nabla g=0$) gives a splitting
\[
TJ(M,g)_j=\calH^\nabla_j\oplus \calV_j
\]
where the horizontal space $\calH^\nabla_j$ is the projection by
$\pi_{1*\xi}$ of the horizontal subspaces in the frame bundle:
$H^\nabla_\xi=\Ker \alpha^\nabla_\xi$ where $\alpha^\nabla$ is the Lie algebra
$\frakg$-valued connection $1$-form on $F(M,g)$ associated to $\nabla$, with
$\frakg=\frako(2p,2q,\bbR)$, $\fraksp(V,\Omega_0)$ or $\frakgl(2n,\bbR)$.

Since $\pi_{*j}\vert_{\calH^\nabla_j}:\calH^\nabla_j\rightarrow
T_{p=\pi(j)}M$  is an isomorphism, this splitting gives an isomorphism of
bundles over $J(M,g)$:
\begin{equation}
TJ(M,g)=\calH^\nabla\oplus \calV\simeq E\oplus \calV
= E \oplus [\Phi, \End (E,g)] \subset E\oplus \End(E,g),
\end{equation}
the projection of $TJ(M,g)$ on $E$ being given by $\pi_*$.

Two natural almost complex structures ${J_{\nabla}^{\pm}}$ are defined on
$J(M,g)$ by:
\begin{equation}
\left.\left({J_{\nabla}^{\pm}}\right)_{j}\right\vert_{\calV_j}(S)
= j\circ S,\qquad
\left.\left({J_{\nabla}^{\pm}}\right)_{j}\right\vert_{\calH^\nabla_j}
=\pm(\left.\pi_{*j}\right\vert_{\calH^\nabla_j})^{-1}\circ j 
\circ (\left.\pi_{*j}\right\vert_{\calH^\nabla_j}).
\end{equation}
In other words,
\[
\left.{J_{\nabla}^{\pm}}\right\vert_{ \calV % \subset \End (E,g)
}
=\Phi\cdot \quad 
\] 
is left multiplication by $\Phi$ on $\calV$ viewed as a subbundle of
$\End(E,g)$ and
\[
{J_{\nabla}^{\pm}}\vert_E=\pm\Phi 
\]
with $\Phi$ as a section of $\End(E,g)$ acting on sections of $E$.

The almost complex structure $J_\nabla^+$ was used by Atiyah et al
\cite{bib:AHS} and the structure $J_\nabla^-$ was introduced by Eells and
Salamon \cite{bib:EellsSalamon} as a first  example of geometrically natural
non-integrable almost complex structure.

 %%%%%%%%%%%%%

\subsection{Pullback connection and projection  on  the vertical bundle $\calV$}
\label{section:projection}

The pullback connection $\pi^{-1}\nabla^E$ on $E$ is induced by the connection
$1$-form $p_2^*\alpha^\nabla$ on the pullback bundle $\pi^{-1}F(M,g)$, with
$$p_2:\pi^{-1}F(M,g)\subset J(M,g)\times F(M,g)\rightarrow F(M,g)$$ the
projection on the second factor. We denote by $p_1$ the bundle projection, i.e.
the projection on the first factor $p_1:\pi^{-1}F(M,g)\subset J(M,g)\times
F(M,g)\rightarrow J(M,g)$.

Now $F(M,g)$ injects in $\pi^{-1}F(M,g)$ via $$i: F(M,g)\rightarrow
\pi^{-1}F(M,g): \xi \mapsto (\pi_1(\xi),\xi)$$ and
$i^*(p_2^*\alpha^\nabla)=\alpha^\nabla$.

The  pullback $E^\prime $  of a vector bundle  associated with $F(M,g)$ for the
representation $\rho$ of $G$ on $W$ (for instance $E^\prime =E$ or $\End(E,g)$) can be
written as,
\[
E^\prime :=\pi^{-1}\left(F(M,g)\times_{G,\rho}W\right)
=\pi^{-1}F(M,g)\times_{G,\rho}W\stackrel{p_1}{\longrightarrow} J(M,g).
\]
A section $s$ of $E^\prime $ can be viewed as the $G$-equivariant function $\tilde{s}$
on the $G$-principal bundle $\pi^{-1}F(M,g)$ with values in $W$ so that
$s(j)=[(j,\xi), \tilde{s}(j,\xi)]$. It is completely determined by its
restriction  ${\widehat{s}}:=i^*\tilde{s}$ defined on $F(M,g)$. Then
\begin{eqnarray}
\wt{ (\pi^{-1}\nabla)^{E^\prime }_{\Xi_j} s}(j, \xi)
&=&\overline{\Xi}_{(j,\xi)}\tilde{s} \,\textrm{ where }\, 
(p_2^*\alpha^\nabla)(\overline{\Xi})=0 
\,\textrm{ and }\,p_{1*}(\overline{\Xi})=\Xi \\
&=&\frac{d}{dt}\tilde{s}(j(t),\xi^\prime (t))\vert_{t=0}
\end{eqnarray}
with $j(t)$ a curve in $J(M,g)$ representing $\Xi_j \in T_jJ(M,g)$ and
$\xi^\prime (t)$ a curve in $F(M,g)$ representing 
$\left(\overline{\pi_{*j}\Xi}\right)_\xi$, the horizontal lift in 
$H_\xi^\nabla\subset T_\xi F(M,g)$ of  ${\pi_{*j}\Xi} \in T_pM$, both curves
projecting on the same curve $p(t)$ in $M$. This implies, since
$X_\xi-\left(\alpha^\nabla_\xi(X_\xi)\right)^*_\xi$ is horizontal, for any
$X_\xi \in T_\xi F(M,g)$, with $A^*$ the fundamental vector field associated to
the right action of $G$ on $F(M,g)$ (i.e. $A^*_\xi=\frac{d}{dt}\xi\circ \exp
tA\vert_{t=0}$) for any $A\in \frakg$, and since $\wt{s}(\xi \exp
tA)=\rho(\exp -tA)\wt{s}(\xi)$:
\begin{eqnarray}
({ \widehat{ (\pi^{-1}\nabla)^{E^\prime }_{\pi_{1*_\xi }X_\xi} s}})(\xi)
&=& (\wt{ (\pi^{-1}\nabla)^{E^\prime }_{\pi_{1*_\xi }X_\xi} s})(\pi_1\xi,\xi)
\nonumber\\
&=& \frac{d}{dt} \wt{  s} (\pi_1\xi(t),\xi(t))\vert_{t=0} 
-\frac{d}{dt} \wt{s}(\pi_1(\xi),\xi \circ 
\exp t\left(\alpha^\nabla_\xi(X_\xi)\right))\vert_{t=0}\nonumber\\
&& \textrm{ where } \xi(t)  \textrm{ is a curve in  } F(M,g)  
\textrm{ representing  } X_\xi \nonumber\\
&=&X_\xi { \widehat{s}} +\rho_*(\alpha^\nabla_\xi(X_\xi)) ({ \widehat{s}}(\xi))
\end{eqnarray}
Observe that the function ${\wt{\Phi}}$ on $\pi^{-1}F(M,g)$ corresponding
to the canonical section $\Phi$ of $\End (E,g)$ is given by
${\wt{\Phi}}(j,\xi)= \xi\circ j\circ\xi^{-1}$ so that its restriction to
$F(M,g)$ is the constant function ${\widehat{\Phi}}(\xi)=\wt{j_0}$. Hence
\begin{equation}
({ \widehat{ (\pi^{-1}\nabla)^{\End(E,g)}_{\,\pi_{1*_\xi }X_\xi} \phi}})(\xi)
=\ad (\alpha^\nabla_\xi(X_\xi)) \wt{j_0}
=[\alpha^\nabla_\xi(X_\xi), \wt{j_0}].
\end{equation}
If $\Xi_j$ is horizontal, we  write $\Xi_j=\pi_{1*_\xi }X_\xi$ with
$\alpha^\nabla_\xi(X_\xi)=0$, so $(\pi^{-1}\nabla)_{\pi_{1*_\xi }X_\xi} \Phi=0$.

If $\Xi_j$ is vertical, we  write $\Xi_j=\pi_{1*_\xi }A^*_\xi $ with $A\in \frakg$
such that $A\wt{j_0}+\wt{j_0} A=0$; then
$\Xi_j=\frac{d}{dt}(\xi\circ \exp tA\circ \wt{j_0} \circ (\xi\circ \exp
tA)^{-1}\vert_{t=0}=\xi \circ [A, \wt{j_0}] \circ \xi^{-1}$, hence
$\widehat{\Xi_j}(\xi)= [A, \wt{j_0}]$ when we view the vertical tangent
vector $\Xi_j$ as an element of $\End(T_pM)=\End(E)_j$. We also have
$\alpha^\nabla_\xi(A^*_\xi))=A$;  hence

\begin{proposition}
The projection on the vertical tangent space $\calV_j$
\begin{equation}
P^{\calV_j}: T_jJ(M,g)
=\calH^\nabla_j\oplus \calV_j\rightarrow \calV_j
=[j,\End(E,g)_j]=[\Phi,\End(E,g)]_j\subset \End(E,g)_j
\end{equation}
is given in terms of the covariant derivative under the pullback connection of
the canonical section $\Phi$ of $\End E$ (defined by \eqnref{eq:defPhi}) via
\begin{equation}
(\pi^{-1}\nabla)^{\End(E,g)}_{\quad \Xi_j } \Phi =P^{\calV_j} (\Xi_j).
\end{equation}
\end{proposition}

Note that we differ here slightly from Proposition 3 in \cite{bib:OB-R}; we
follow their development, adapting to  this difference.

Recall that the projection on $\calH^\nabla$ identified with $E$ is given
by $\pi_*$.%\\[2mm]

%%%%%%%%%%
\subsection{A connection on $TJ(M,g)$ preserving $J_{\nabla}^{\pm}$}
\label{section:connection}

We define a covariant derivative of sections of  $E$ preserving $g$ so that the
associated covariant derivative of sections of $\End(E,g)$ preserves sections of
$\calV$; let
\begin{equation}
D^E_\Xi Y:= (\pi^{-1}\nabla)^E_\Xi Y +\half (P^{\calV}(\Xi)\circ \Phi)(Y),  
\qquad \Xi\in \Gamma(TJ(M,g)),\quad Y\in\Gamma(E)
\end{equation}
where  $\Phi$ and $P^{\calV}(\Xi)$ are viewed as sections of $\End(E,g)$.
This covariant derivative preserves the tautological complex structure $\Phi$ on
the bundle $E$ since it is equal to
\[
D^E=(\pi^{-1}\nabla)^E-\half \Phi\circ \left((\pi^{-1}\nabla)^{\End E}\Phi\right).
\]
The associated covariant derivative of sections of $\End E$ is given by
\begin{equation}
D^{\End E}_\Xi S:=(\pi^{-1}\nabla)^{\End E}_\Xi S 
+\half [P^{\calV}(\Xi)\circ \Phi,S], 
\qquad S \in \Gamma(\End E).
\end{equation}
Since  $\Phi$ anticommutes with any element of $\calV$, and $\Phi^2=-\id$
we have indeed
\begin{equation}
D^{\End E}_\Xi \Phi=(\pi^{-1}\nabla)^{\End E}_\Xi \Phi
+\half [P^{\calV}(\Xi)\circ \Phi,\Phi] 
= P^{\calV}(\Xi)-P^{\calV}(\Xi)=0.
\end{equation}
Hence $D^{\End E}$ preserves sections of $\calV=[\Phi,\End(E,g)]$ and
$D^E\oplus D^{\End E}$ induces a covariant derivative $D$ of sections of the
tangent bundle $TJ(M,g)$. If Y is a section of $\calH^\nabla\simeq E$,
then
\[
D_\Xi ({J_{\nabla}^{\pm}} Y)=D_\Xi(\pm\Phi(Y))
=\pm\Phi(D_\Xi(Y)
={J_{\nabla}^{\pm}} D_\Xi(Y);
\]
if $S$ is a section of $\calV\subset\End(E,g)$ then
\[
D_\Xi ({J_{\nabla}^{\pm}} S)=D_\Xi(\Phi\circ S)
=\Phi\circ D_\Xi(S)={J_{\nabla}^{\pm}} D_\Xi(S).
\]
Hence \begin{equation}
D{J_{\nabla}^{\pm}}=0.
\end{equation}
Since $D$ preserves $\calH^\nabla\simeq E$ and $\calV$, the
covariant derivative of the projections  vanish:
\begin{equation}
DP^{\calV}=0 \qquad D\pi_*=0.
\end{equation}

 %%%%%%%%%%%%%%%%%%%

\section{A closed $2$-form on $J(M,g)$ associated to $\nabla$}
\label{section:2form}

Observe that $D^E$ preserves the tautological  complex structure defined by
$\Phi$ on the bundle $E$, hence can be used, following Reznikov
\cite{bib:Reznikov} and Rawnsley \cite{bib:Rawnsley}, in the Chern--Weil
construction of characteristic classes of $E$;  the complex trace of the
curvature  of $D^E$,
\[
\chi(J(M,g)) \ni \Xi, \Xi^\prime  \mapsto  \Tr_\bbC \left(D^E_{\Xi}\circ D^E_{\Xi^\prime }-
D^E_{\Xi^\prime }\circ D^E_{\Xi}-D^E_{[\Xi,\Xi^\prime ]}\right)
\]
is $-2\pi \sqrt{-1}$ times a real closed $2$-form on $J(M,g)$ representing
$c_1(E,\Phi) \in H^2(J(M,g),\bbR)$ which is the real first Chern class of the complex
vector bundle $(E,\Phi)$.

\begin{proposition}\cite{bib:Reznikov}\label{prop:symplectic}
Having chosen a torsion-free connection $\nabla$  preserving the
pseudo-Rie\-mann\-ian  or symplectic structure $g$, the $2$-form
$\omega^{J(M,g,\nabla)}$ on $J(M,g)$ defined by
\begin{equation}\label{def:omegaJ(M,g,nabla)}
\omega^{J(M,g,\nabla)}_j(\Xi,\Xi^\prime ):=
-2\Tr_\bbR \left(R^\nabla_{\pi(j)} (\pi_{*j}\Xi,\pi_{*j}\Xi^\prime )\circ j\right) 
+i\Tr_\bbC\left(\left[ P^\calV(\Xi),P^\calV(\Xi^\prime )  \right]\right),
\end{equation}
which represents $-8\pi c_1(E,\Phi)$, is symplectic if and only if, for any
$p\in M$ and any  $j\in J(M,g)_p$, the skew-symmetric bilinear form
$\Omega^{\nabla,j}$ on $T_pM$
\begin{equation}\label{eq:}
X,Y \mapsto \Tr_\bbR (R^\nabla_p (X,Y)\circ j)
\end{equation}
is non-degenerate.
\end{proposition}
\begin{proof}
 Indeed, since $D^E=(\pi^{-1}\nabla)^E
 -\half \Phi\circ \left((\pi^{-1}\nabla)^{\End E}\Phi\right)$, 
 we have
\begin{eqnarray*}
\left(D^E_{\Xi}\circ D^E_{\Xi^\prime }-D^E_{\Xi^\prime }\circ D^E_{\Xi}
-D^E_{[\Xi,\Xi^\prime ]}\right)
&=&\pi^*(R^\nabla(\pi_*\Xi,\pi_*\Xi^\prime ))
-\half \Phi\left[\pi^*(R^\nabla(\pi_*\Xi,\pi_*\Xi^\prime )),
\phi\right] \\
&&\quad\quad-\frac{1}{4}  
\left[(\pi^{-1}\nabla)^{\End E}_\Xi\Phi,(\pi^{-1}\nabla)
^{\End E}_{\Xi^\prime }\Phi \right]\\
&=&\half\pi^*(R^\nabla(\pi_*\Xi,\pi_*\Xi^\prime ))
-\half \Phi\, \pi^*(R^\nabla(\pi_*\Xi,\pi_*\Xi^\prime ))\, 
\Phi\\
&&\quad\quad-\frac{1}{4}  \left[P^\calV(\Xi),P^\calV(\Xi^\prime ) \right]
\end{eqnarray*}
where
$(\pi^*(R^\nabla(\pi_*\Xi,\pi_*\Xi^\prime )))(j):=R^\nabla_{p=\pi(j)}(\pi_{*j}\Xi_j,
\pi_{*j}\Xi^\prime _j)$ is viewed as an endomorphism of $T_pM$, hence as an element of
$\End(E,g)_j$.

Observe that $V_1,V_2\in T_j(\calV_j)\rightarrow
i\Tr_\bbC\left(\left[V_1,V_2 \right]\right)_j=-\Tr_\bbR (j\left[V_1,V_2 \right])$
defines the usual symplectic structure on the fibre of $J(M,g)$, i.e. the one
induced by the isomorphism between a fibre and $Gl(V,\wt{g}_0)/
Gl(V,\wt{g}_0,\wt{j_0})$.

Hence the closed $2$-form $\omega^{J(M,g)}(\Xi,\Xi^\prime ) =-4i\Tr_\bbC
\left(D^E_{\Xi}\circ D^E_{\Xi^\prime } -D^E_{\Xi^\prime }\circ
D^E_{\Xi}-D^E_{[\Xi,\Xi^\prime ]}\right)$ is symplectic  if and only if, for any $p\in
M$ and any  $j\in J(M,g)_p$, the bilinear form on $T_pM$ $X,Y \mapsto \Tr_\bbR
(R^\nabla_p (X,Y)\circ j)$ is non-degenerate.
\end{proof}

\begin{lemmadef}
Each of the almost complex structures $J^\pm_\nabla$ is said to be 
compatible with the closed $2$-form $\omega^{J(M,g,\nabla)}$ when
\begin{equation}\label{eq:compat}
\omega^{J(M,g,\nabla)}(J^\pm_\nabla\Xi,J^\pm_\nabla \Xi^\prime)
=\omega^{J(M,g,\nabla)}(\Xi,\Xi^\prime)
\end{equation} 
i.e.\  when $\omega^{J(M,g,\nabla)}$ is of type $(1,1)$ with respect to $J^\pm_\nabla$.
This will be true  if and only if
\begin{eqnarray}
\kern-1cm\Tr_\bbR \left(R^\nabla_{p} (jX, jY)\circ j\right)&=&\Tr_\bbR \left(R^\nabla_{p} 
(X, Y)\circ j\right),\nonumber\\ 
&& \qquad \forall p\in M,\, X,Y \in T_pM, j\in J(M,g)_p.\label{eq:Appendix}
\end{eqnarray} 
\end{lemmadef}

%%%%%%%%%%%%%%%%%%%%%%%%

\section{The Nijenhuis tensor for $J_{\nabla}^{\pm}$}
\label{section:Nijenhuis}

The torsion $T^D$ of $D$ can be used to compute the Nijenhuis tensor of $J^\pm_\nabla$.
Now the vertical part of the torsion $T^D$ is given by
\begin{eqnarray*}
P^{\calV} T^D(\Xi,\Xi^\prime )
&=&P^{\calV}\left(D_\Xi\Xi^\prime -D_{\Xi^\prime }\Xi-[\Xi,\Xi^\prime ]\right)\\
&=&D_\Xi(P^{\calV}(\Xi^\prime ))-D_{\Xi^\prime }(P^{\calV}(\Xi))
-P^{\calV}([\Xi,\Xi^\prime ])\\
&=&D_\Xi((\pi^{-1}\nabla)^{\End E}_{\Xi^\prime } \Phi)-D_{\Xi^\prime }
((\pi^{-1}\nabla)^{\End E}_{\Xi} \Phi)-(\pi^{-1}\nabla)^{\End E}_{[\Xi,\Xi^\prime ]} \Phi\\
&=&(\pi^{-1}\nabla)^{\End E}_{\Xi} ((\pi^{-1}\nabla)^{\End E}_{\Xi^\prime } \Phi)+\half 
[P^{\calV}(\Xi)\circ \Phi,P^{\calV}(\Xi^\prime )]\\
&&- (\pi^{-1}\nabla)^{\End E}_{\Xi^\prime } ((\pi^{-1}\nabla)^{\End E}_{\Xi} \Phi)-\half 
[P^{\calV}(\Xi^\prime )\circ \Phi,P^{\calV}(\Xi)]\\
&&-(\pi^{-1}\nabla)^{\End E}_{[\Xi,\Xi^\prime ]} \Phi
=[\pi^*(R^\nabla(\pi_*\Xi,\pi_*\Xi^\prime )),\Phi]
+\frac{1}{4}\left[[P^{\calV}(\Xi),P^{\calV}(\Xi^\prime )],\Phi\right]\\
&=&[\pi^*(R^\nabla(\pi_*\Xi,\pi_*\Xi^\prime )),\Phi]
\end{eqnarray*}
where
$(\pi^*(R^\nabla(\pi_*\Xi,\pi_*\Xi^\prime )))(j):=R^\nabla_{p=\pi(j)}(\pi_{*j}\Xi_j,
\pi_{*j}\Xi^\prime _j)$ is viewed as an endomorphism of $T_pM$ hence as an element of
$\End(E,g)_j$. The horizontal part of the torsion is given by
\begin{eqnarray*}
\pi_*T^D(\Xi,\Xi^\prime )
&=& \pi_*\left(D_\Xi\Xi^\prime -D_{\Xi^\prime }\Xi-[\Xi,\Xi^\prime ]\right)
=D_\Xi(\pi_*\Xi^\prime )-D_{\Xi^\prime }(\pi_*\Xi)-\pi_*([\Xi,\Xi^\prime ])\\
&=& (\pi^{-1}\nabla)^{ E}_{\Xi} (\pi_*\Xi^\prime )
-(\pi^{-1}\nabla)^{ E}_{\Xi^\prime } (\pi_*\Xi)-\pi_*([\Xi,\Xi^\prime ])\\
&&+\half(P^{\calV}(\Xi)\circ \Phi)(\pi_*\Xi^\prime )
-\half(P^{\calV}(\Xi^\prime )\circ \Phi)(\pi_*\Xi)\\
&=&\pi^* \left(T^\nabla(\pi_*\Xi,\pi_*\Xi^\prime )\right)
-\half \Phi\left(P^{\calV}(\Xi)(\pi_*\Xi^\prime )
-P^{\calV}(\Xi^\prime )(\pi_*\Xi)\right)
\end{eqnarray*}
where $\left(\pi^*
\left(T^\nabla(\pi_*\Xi,\pi_*\Xi^\prime )\right)\right)(j):=T^\nabla_{p=\pi(j)}(\pi_{*j
}\Xi_j,\pi_{*j}\Xi^\prime _j)$ is an element of $T_pM$ viewed as an element of $E_j$.

Since $D{J_{\nabla}^{\pm}}=0$ we know that 
\begin{equation}\label{eq:torsionN}
T^{\nabla^\prime }(JX,JY)-JT^{\nabla^\prime }(JX,Y)-JT^{\nabla^\prime }(X,JY)
-T^{\nabla^\prime }(X,Y)=-N^J(X,Y)
\end{equation}
hence
\[
N^{{J_{\nabla}^{\pm}}}(\Xi,\Xi^\prime )
=-T^D({J_{\nabla}^{\pm}}\Xi,{J_{\nabla}^{\pm}}\Xi^\prime )
+{J_{\nabla}^{\pm}}T^D({J_{\nabla}^{\pm}}\Xi,\Xi^\prime )
+{J_{\nabla}^{\pm}}T^D(\Xi,{J_{\nabla}^{\pm}}\Xi^\prime )+T^D(\Xi,\Xi^\prime ).
\]
From the formulas above, since $ \pi_*({J_{\nabla}^{\pm}}
\Xi)=\pm\Phi(\pi_*\Xi)$ and $P^{\calV}({J_{\nabla}^{\pm}} \Xi))=\Phi\circ
P^{\calV}(\Xi)$, we get

\begin{eqnarray}
P^{\calV}(N^{{J_{\nabla}^{\pm}}}(\Xi,\Xi^\prime ))
&=&-[\pi^*(R^\nabla(\Phi(\pi_*\Xi),\Phi(\pi_*\Xi^\prime ))),\Phi]\pm \Phi\circ 
[\pi^*(R^\nabla(\Phi(\pi_*\Xi),\pi_*\Xi^\prime )),\Phi]\nonumber\\
&& \pm \Phi\circ [\pi^*(R^\nabla(\pi_*\Xi,\Phi(\pi_*\Xi^\prime ))),\Phi] 
+[\pi^*(R^\nabla(\pi_*\Xi,\pi_*\Xi^\prime )),\Phi]\\
\pi_*N^{{J_{\nabla}^{\pm}}}(\Xi,\Xi^\prime )
&=& -\pi^* \left(T^\nabla(\Phi(\pi_*\Xi),\Phi(\pi_*\Xi^\prime ))\right)
+\Phi\left(\pi^* \left(T^\nabla(\Phi(\pi_*\Xi),\pi_*\Xi^\prime )\right)\right)\nonumber\\
&& +\Phi\left(\pi^* \left(T^\nabla(\pi_*\Xi,\Phi(\pi_*\Xi^\prime ))\right)\right)
+\pi^* \left(T^\nabla(\pi_*\Xi,\pi_*\Xi^\prime )\right)\nonumber\\
&&\pm\half \Phi\left((\Phi\circ P^{\calV}(\Xi))(\Phi\pi_*\Xi^\prime )
-(\Phi\circ P^{\calV}(\Xi^\prime ))(\Phi\pi_*\Xi)\right)\nonumber\\
&&\mp\half \Phi^2\left(\Phi\circ P^{\calV}(\Xi)(\pi_*\Xi^\prime )
-\Phi\circ P^{\calV}(\Xi^\prime )(\pi_*\Xi)\right)\nonumber\\
&&-\half \Phi^2\left(P^{\calV}(\Xi^\prime )(\Phi(\pi_*\Xi))
-P^{\calV}(\Xi)(\Phi(\pi_*\Xi^\prime ))\right)  \nonumber \\
&&-\half \Phi\left(P^{\calV}(\Xi)(\pi_*\Xi^\prime )
-P^{\calV}(\Xi^\prime )(\pi_*\Xi)\right)\nonumber\\
&=& -\pi^* \left(T^\nabla(\Phi(\pi_*\Xi),\Phi(\pi_*\Xi^\prime ))\right)
+\Phi\left(\pi^* \left(T^\nabla(\Phi(\pi_*\Xi),\pi_*\Xi^\prime )\right)\right)\nonumber\\
&& +\Phi\left(\pi^* \left(T^\nabla(\pi_*\Xi,\Phi(\pi_*\Xi^\prime ))\right)\right)
+\pi^* \left(T^\nabla(\pi_*\Xi,\pi_*\Xi^\prime )\right)\nonumber\\
&&(\mp1+1) \left((P^{\calV}(\Xi))(\Phi\pi_*\Xi^\prime )
-(P^{\calV}(\Xi^\prime ))(\Phi\pi_*\Xi)\right)
\end{eqnarray}

\begin{proposition}
The Nijenhuis tensor associated to the canonical almost complex structures
${J_{\nabla}^{\pm}}$ on the twistor space $J(M,g)$ always vanishes on two
vertical vector fields; $N^{J^{ +}_{\nabla}}$ vanishes on $
{\calV}\times{\calH}^\nabla$ whereas $J_{\nabla}^{ -}$ is never
integrable because
\begin{equation}
N^{J_{\nabla}^{ -}}_j(S,Y)= 2 S_j(jY_j)=-2jS_j(Y_j), 
\quad \begin{array}{l} \textrm{ for } 
S\in \Gamma(\calV)\subset \Gamma(\End(E,g))\\
\, \textrm{ and } Y\in \Gamma(\calH^\nabla)=\Gamma(E) \end{array}
\end{equation}
so that $\Image N^{J_{\nabla}^{ -}}\supset \calH^\nabla$.

Choosing the connection $\nabla$ without torsion (which will be the Levi Civita
connection in the pseudo-Riemannian setting) one sees that the horizontal part
of $N^{{J_{\nabla}^{\pm}}}$ vanishes on $
{\calH}^\nabla\times{\calH}^\nabla$, hence $\Image N^{J_\nabla^+}
\subset \calV$.

The vertical part of the image of $N^{{J_{\nabla}^{\pm}}}_j$ consists of all the
endomorphisms of $T_pM$ with $p=\pi(j)$ given by
\begin{eqnarray*}
-[R^\nabla_p(jXjX^\prime ),j]\pm j\circ [R^\nabla_p(j X,X^\prime ),j] 
\pm j \circ [R^\nabla_p(X, jX^\prime ),j]
 +[R^\nabla_p(X,X^\prime ),j]\nonumber \\
&&\kern-5in =j\circ R^\nabla_p(jXjX^\prime )-R^\nabla_p(jXjX^\prime )\circ j \pm j\circ 
R^\nabla_p(j X,X^\prime )\circ j \pm R^\nabla_p(j X,X^\prime )\nonumber \\
&&\kern-4.6in \pm j\circ R^\nabla_p(X,jX^\prime )\circ j \pm R^\nabla_p(X,jX^\prime ) 
+R^\nabla_p(X,X^\prime )\circ j -j\circ R^\nabla_p(X,X^\prime ) \nonumber \\
&&\kern-5in =\textrm{Imaginary part of } \left(\Id-ij)\circ 
R^\nabla_p\left((\id\pm ij)X,(\id \pm ij) X^\prime \right)\circ (\id+ij)\right)
\end{eqnarray*}
which is equal to $\textrm{Real part of }-j\left(\Id-ij)\circ 
R^\nabla_p\left((\id\pm ij)X,(\id \pm ij) X^\prime \right)\circ (\id+ij)\right)$.
\end{proposition}
We now proceed as in \cite{bib:OB-R}: the vertical part of the image of
$N^{{J_{\nabla}^{\pm}}}_j$ vanishes identically  on all $j' s \in \pi^{-1}p$ if
and only if the curvature $\tilde R$, which is the expression (using a frame) of
$R^\nabla_p$ as a $1,3$ tensor on $V$, satisfies
\[
(\Id-i{\wt{j}})\circ \tilde{R}\left((\id\pm i{\wt{j}})\, 
\cdot \, ,(\id \pm i{\wt{j}}) \,\cdot \, \right)
\circ (\id+i{\wt{j}})=0, \quad \forall {\wt{j}}=A{\wt{j_0}}A^{-1}
\]
where $A \in G$. Hence for all $A \in G$ and putting $\wt{j}=A{\wt{j_0}}A^{-1}$
\begin{eqnarray*}
&&A(\Id-i{\wt{j_0}})A^{-1}\circ \tilde{R}\left(A(\id\pm i{\wt{j_0}})A^{-1}\, 
\cdot \, ,A(\id \pm i{\wt{j_0}})A^{-1} \,\cdot \, \right)
\circ A(\id+i{\wt{j_0}})A^{-1}=0,\,\mathrm{so}\\
&&(\Id-i{\wt{j_0}})\circ A^{-1} \tilde{R}\left(A(\id\pm i{\wt{j_0}})\, 
\cdot \, ,A(\id \pm i{\wt{j_0}}) \,\cdot \, \right)
 A \circ(\id+i{\wt{j_0}})=0, \quad \mathrm{so}\\
&&(\Id-i{\wt{j_0}})\circ A^{-1}\cdot\tilde{R}\left((\id\pm i{\wt{j_0}})\, 
\cdot \, ,(\id \pm i{\wt{j_0}}) \,\cdot \, \right)
\circ (\id+i{\wt{j_0}})=0 , 
\end{eqnarray*}
where  $A^{-1}\cdot\tilde{R}:= A^{-1}\tilde{R}(A \cdot,A\cdot)A$ denotes the
natural action of $G$ on tensors, hence if and only if the curvature $\tilde R$
takes values in the largest $G$-invariant subspace of  tensors on $V$ of
(pseudo-Riemannian, symplectic or plain) curvature type for which
\begin{equation}\label{eq:j0condition}
(\Id-i{\wt{j_0}})\circ \tilde{R}\left((\id\pm i{\wt{j_0}})\, 
\cdot \, ,(\id \pm i{\wt{j_0}}) \,\cdot \, \right)
\circ (\id+i{\wt{j_0}})=0.
\end{equation}
There is a natural action of ${\wt{j_0}}$ on curvature type tensors given by
\[
({\wt{j_0}}\cdot \tilde{R})(U,V)={\wt{j_0}}\circ \tilde{R}(U,V) 
-\tilde{R}({\wt{j_0}}U,V)-\tilde{R}(U,{\wt{j_0}}V)
-\tilde{R}(U,V)\circ {\wt{j_0}}.
\]
The action of ${\wt{j_0}}$ on $V^\bbC$ has $\pm i$ as eigenvalues,  the
projection on the $+i$-eigenspace being given by $\Id-i{\wt{j_0}}$. Hence
the action on the space of tensors of curvature type has eigenvalues in $\{0,\pm
2i,\pm 4i\}$; the projection on the $4i$-eigenspace is given by
\[
(\Id-i{\wt{j_0}})\circ \tilde{R}\left((\id+ i{\wt{j_0}})\, 
\cdot \, ,(\id + i{\wt{j_0}}) \,\cdot \, \right)\circ (\id+i{\wt{j_0}}),
\]
thus \eqnref{eq:j0condition} says that the vertical part of the image of
$N^{{J_{\nabla}^+}}$ vanishes if and only if $\tilde R$  takes values in the
largest $G$-invariant subspace of  curvature-type tensors on $V$ for which $4i $
is not an eigenvalue of the action of $\wt{j_0}$.

Next we examine the decomposition of the space of curvature type tensors under the
action of $G$.

%%%%%%%%%%%%%%%%%%%%%%%%%%%

\section{Pseudo-Riemannian structure of signature $(2p,2q)$\\
with (or with\-out) a given orientation}

In the case of a pseudo-Riemannian structure $g$ of signature $(2p,2q)$ on a
manifold $M$, one uses the Levi Civita connection for $\nabla$. 

\begin{definition}
The space of curvature type tensors at the point $p\in M$,
\[
\bigg\{ R\in \Lambda^2(V^*)\otimes\End(V) \bigg\vert \, \cyclic_{X,Y,Z}
R(X,Y)Z=0,\ g_p(R(X,Y)Z,T)=-g_p(R(X,Y)T,Z)\bigg\},
\]
with $V:=T_pM$, will be denoted by $\CC(V,g_p)$  where
$\cyclic_{X,Y,Z}\wt{R}(X,Y)Z$ here and elsewhere denotes the sum over cyclic
permutations of $X,Y,Z$.
\end{definition}
When $G=O(2p,2q)$ with $2p+2q=2n$, this space of  curvature type tensors splits
into $3$-irreducible parts \cite{bib:Besse} so that: 
$$R^\nabla=S^\nabla+E^\nabla+C^\nabla,$$ where $S^\nabla$ is  constructed
algebraically using the metric tensor $g$ and the scalar curvature
$scal(g)=\Tr\rho^\nabla$  with $g(X,\rho^\nabla
Z):=Ric^\nabla(X,Z):=\Tr[Y\rightarrow R^\nabla(X,Y)Z] $ 
\[
g(S^\nabla(X,Y)Z,T)=\frac{ scal(g)}{2n(2n-1)}
\left(g(X,Z)g(Y,T)-g(X,T)g(Y,Z)\right),
\]
where $E^\nabla$ is the half traceless part constructed algebraically using the
metric tensor and the traceless part of the Ricci tensor
(${\widehat{Ric}}(X,Z)=Ric^\nabla(X,Z)-\frac{ scal(g)}{2n} g(X,Z)$):
\begin{eqnarray*} 
g(E^\nabla(X,Y)Z,T)%&&\\&&\kern-1in \mbox{}
&=&{\frac {1}{2n-2}}\,\left(g(X,Z){\widehat{Ric}}(Y,T)
-g(X,T){\widehat{Ric}}(Y,Z)\right.\\
&& %\kern-0.7in\left.\mbox{}
~\qquad \qquad + \left. g(Y,T){\widehat{Ric}}(X,Z)
-g(Y,Z){\widehat{Ric}}(X,T)\right)  
\end{eqnarray*}
and where $C^\nabla$ is the totally traceless part, the so-called Weyl tensor.

Since
${\wt{g_0}}({\wt{j_0}}X,Y)+{\wt{g_0}}(X,{\wt{j_0}}Y) =0$, the 4i eigenvalue can
only arise in the Weyl tensor part and does so, hence the well known
\begin{proposition}\label{prop:pseudoRnonor}
 $J^+_\nabla$ is integrable in the pseudo-Riemannian context with no given
orientation if and only if $C^\nabla=0$.
\end{proposition}

In the oriented case the decomposition of the curvature under the action of
$SO(2p,2q)$ is the same as above in dimension greater than $4$ but in dimension
$4$, there is a further splitting  of the Weyl tensor into a self-dual and an
anti-self-dual part. A Weyl tensor is said to be self-dual (respectively
anti-self-dual), if, viewed as a endomorphism of $\Lambda^2T^*M$, it vanishes on
the eigenspace of eigenvalue $-1$ (respectively $+1$) of the Hodge $*$ operator
acting on 2-forms.

\begin{proposition}\label{prop:pseudoRor}
$J^+_\nabla$ is integrable in the pseudo-Riemannian context with a  given
orientation if and only if $C^\nabla=0$ when $2n\ge 4$;in dimension $4$,it is
integrable if and only if the the Weyl component of the Riemann curvature tensor
is self-dual when the signature is $(4,0)$ or $(0,4)$ and anti-self-dual when
the signature is $(2,2)$. 
\end{proposition}

\begin{proof} [Proof (in dimension $4$).]
In an oriented pseudo-orthonormal basis $\{e_1,\ldots,e_4\}$ with \\
${\wt{g_0}}(e_1,e_1)={\wt{g_0}}(e_3,e_3)=\epsilon_1$
and ${\wt{g_0}}(e_2,e_2)={\wt{g_0}}(e_4,e_4)=\epsilon_2$ and with 
${\wt{j_0}}=\begin{pmatrix} 0&-\id_2\\ \id_2 & 0\end{pmatrix}$ as before,
 the eigenspace of eigenvalue $\epsilon$  of the Hodge $*$ operator is spanned 
 by $e_1\wedge e_2+ \epsilon \epsilon_1\epsilon_2\,\, e_3\wedge e_4=e_1\wedge e_2 + \epsilon \epsilon_1\epsilon_2\,\,{\wt{j_0}}e_1\wedge {\wt{j_0}}e_2,$   
 $\, e_1\wedge e_3- \epsilon \,\, e_2\wedge e_4$ and 
 $e_1\wedge e_4+ \epsilon \epsilon_1\epsilon_2\,\, e_2\wedge e_3=e_1\wedge e_4 + \epsilon \epsilon_1\epsilon_2\,\,{\wt{j_0}}e_1\wedge {\wt{j_0}}e_4$. Hence, any tensor $\tilde{R}$ 
 vanishing on  the eigenspace of eigenvalue $\epsilon=-\epsilon_1 \epsilon_2$
 satisfies  $\tilde{R}({\wt{j_0}}\, 
\cdot \, ,{\wt{j_0}}\, 
\cdot \,)= \tilde{R}(\, 
\cdot \,,\, 
\cdot \,)$, hence $\tilde{R}\left((\id+ i{\wt{j_0}})\, 
\cdot \, ,(\id + i{\wt{j_0}})\, 
\cdot \,\right)=0$.
The largest $SO(V,g)$-invariant subspace of  Weyl  tensors on $V$ for which $4i$
is not an eigenvalue of the action of $\wt{j_0}$ is thus the space of Weyl
tensors vanishing on  the eigenspace of eigenvalue
$\epsilon=-\epsilon_1\epsilon_2$  of the Hodge $*$ operator. 
\end{proof}

%\smallskip

Observe that
\begin{eqnarray*}
g\left((\Id-ij)S^\nabla\left((\id- ij)X,(\id - ij) Y\right) (\id+ij)Z,T\right)\\
&&\kern-3.5in =\frac{ 2 \, scal(g)}{n(2n-1)}\left(g((\id- ij)X,Z) 
g((\id- ij)Y,T)-g((\id- ij)Y,Z)g((\id- ij)X,T)\right)
\end{eqnarray*}
hence
\begin{eqnarray*}
\textrm{Imaginary part of } \left(\Id-ij)\circ S^\nabla\left((\id- ij)X,
(\id - ij) X^\prime \right)\circ (\id+ij)\right)\\
&&\kern-3.6in =
\frac{ 2scal(g)}{n(2n-1)} \left(g(X^\prime ,\cdot)jX + g(jX^\prime ,
\cdot)X- g(X,\cdot)jX^\prime  - g(jX,\cdot)X^\prime \right)\\
&&\kern-3.6in =
\frac{ 2scal(g)}{n(2n-1)} \left[\,j\, , 
g(X^\prime ,\cdot)X - g(X,\cdot)X^\prime  \, \right];
\end{eqnarray*}
and this shows that the vertical part of the image of $\ N^{J_\nabla^-}$ at $j$
is the whole vertical tangent space
$\calV_j=[j,\End(E,g)_j]=[j,\End(T_pM,g_p)]$ whenever the space has
constant non-zero sectional curvature, i.e. when $R^\nabla=S^\nabla$ and
$scal(g)\neq 0$.

To summarise, we have

\begin{proposition}\label{prop:pseudoRi}
For  a pseudo-Riemannian manifold $(M,g)$ with no given orientation, the almost
complex structure $J^+_\nabla$ on the twistor space $J(M,g)$,  defined using the
Levi Civita connection $\nabla$, is integrable if and only if the Weyl component
of the Riemann curvature tensor vanishes, $C^\nabla=0$.\\
With a given orientation,  the almost complex structure $J^+_\nabla$ on the
twistor space $J(M,g)$,  defined using the Levi Civita connection $\nabla$, is
integrable if and only if the Weyl tensor $C^\nabla$ vanishes when $\dim M >4$.
In dimension $4$,it is integrable if and only if the the Weyl component of the
Riemann curvature tensor is self-dual when the signature is $(4,0)$ or $(0,4)$
and anti-self-dual when the signature is $(2,2)$. 

The almost complex structure $J^-_\nabla$ is never integrable.

If the space has non-vanishing constant sectional curvature, then the image of
the Nijenhuis tensor associated to $J_\nabla^-$ is the whole tangent space
$T_jJ(M,g)$ at any point $j\in J(M,g)$.

Observe that in this case ($C^\nabla=0$, $E^\nabla=0$ and $scal(g)\neq 0$),  the
closed $2$-form on $J(M,g)$ associated by  \eqnref{def:omegaJ(M,g,nabla)} to
$\nabla$, $\omega^{J(M,g,\nabla)}$, is symplectic since $\Tr (R^\nabla(X,Y)\circ
j)=\frac{scal(g)}{n(2n-1)}g(X,jY)$. Also in that case, the almost complex
structures $J^\pm_\nabla$ are compatible with the symplectic $2$-form, in the
sense of equation \eqnref{eq:compat}, i.e. $\omega^{J(M,g,\nabla)}$ is of type
$(1,1)$ with respect to $J^\pm_\nabla$; $J^+_\nabla$ is positive when $scal(g)$
is positive and $J^-_\nabla$ is positive when $scal(g)$ is negative.
\end{proposition}

Hence the twistor space $J(M,g)$ on a pseudo-Riemannian manifold with
non-vanishing constant sectional curvature has a natural symplectic structure
$\omega^{J(M,g,\nabla)}$ and two  natural compatible almost complex structures,
$J^+_\nabla$ yielding a pseudo-K\"ahler structure on this twistor space and
$J^-_\nabla$ being maximally non-integrable in the sense that the image of the
corresponding Nijenhuis tensor is the whole tangent space at every point.
\vskip1cm

More generally, for the twistor space on a Riemannian space, Reznikov
\cite{bib:Reznikov}  has proven that the closed $2$-form
$\omega^{J(M,g,\nabla)}$ (defined by \eqnref{def:omegaJ(M,g,nabla)}) is
symplectic if the sectional curvature is sufficiently pinched. The proof relies
on  Berger's inequalities \cite{bib:Berger}, all components
$R_{ijk\ell}:=g_p(R_p(e_i,e_j)e_k,e_l) $ of the curvature tensor in an
orthonormal basis $\{e_i;i\le 2n\}$ of $T_pM$ are very small unless
$\{i,j\}=\{k,l\}$. Hence the $2$-form $X,Y\mapsto R^\nabla(X,Y)\circ j$ is very
close to the $2$-form $X,Y\mapsto\frac{scal(g)}{n(2n-1)}g(X,jY)$ and is thus 
non-degenerate.

In a similar way, the endomorphism of $T_pM$ defined by 
\[
\textrm{Imaginary part of } \left(\Id-ij)\circ R^\nabla_p\left((\id\pm ij)X,
(\id \pm ij) X^\prime \right)\circ (\id+ij)\right)
\] 
is very close to $ \frac{ 2scal(g)}{n(2n-1)} \left[\,j\, , g(X^\prime ,\cdot)X -
g(X,\cdot)X^\prime  \, \right]$ hence the vertical part of the image of
$N^{{J_{\nabla}^-}}_j$ consists of all the endomorphisms  $[j, A]$ of $T_pM$
where $p=\pi(j)$ and $A\in \End(T_pM,g_p)$.

\begin{proposition}\label{prop:pinchedcurv}
Given any positive integer $n$, there exists an $\epsilon(n)$ such that, if the
sectional curvature of a Riemannian manifold $(M,g)$ of dimension $2n$ is
$\epsilon(n)$-pinched, the almost complex structure $J^-_\nabla$ on this twistor
space, defined using the Levi Civita connection $\nabla$, is maximally
non-integrable (i.e. the image of the corresponding Nijenhuis tensor is the
whole tangent space at every point).
\end{proposition}

We shall now study when  each of the almost complex structures $J^\pm_\nabla$ is
compatible (in the classical sense of equation (\ref{eq:compat})) with the
$2$-form  $\omega^{J(M,g,\nabla)}$ (defined by equation
\eqnref{def:omegaJ(M,g,nabla)}); we have seen in Section \ref{section:2form}
that it is the case  if and only if equation \eqnref{eq:Appendix} is satisfied:
$ \Tr_\bbR \left(R^\nabla_{p} (jX, jY)\circ j\right)=\Tr_\bbR \left(R^\nabla_{p} (X,
Y)\circ j\right)$ for all $ p\in M,\, X,Y \in T_pM, j\in J(M,g)_p.$

\begin{definition}\label{def:Omega1}
For $R \in \CC(V,g_p)$ and $j \in J(M,g)_p$ let $\Omega^{R,j}_1(X,Y) = \Tr_\bbR
\left(R (X, Y)\circ j\right)$ for $X,Y \in V$. 
\end{definition}

The condition of compatibility
\eqnref{eq:Appendix} is that $\Omega^{R,j}_1(jX,jY) = \Omega^{R,j}_1(X,Y)$ for
all $X,Y \in V$ so if we define 

\begin{definition}\label{def:Omega2}
 $\Omega^{R,j}_2(X,Y) =
\Omega^{R,j}_1(jX,jY)-\Omega^{R,j}_1(X,Y)$, 
\end{definition}

\noindent then the condition for compatibility becomes
$\Omega^{R,j}_2=0$ for all $j\in J(M,g)_p$.

\begin{proposition} \label{prop:Appendix}
Let $(M,g)$ be a pseudo-Riemannian manifold of dimension $2n \ge4$ with Levi
Civita connection $\nabla$. Condition \eqnref{eq:Appendix} holds (i.e. $J^\pm_\nabla$
are compatible with the closed 2-form
$\omega^{J(M,g,\nabla)}$) for $M$
non-oriented and $2n\ge4$ or $M$ oriented and $2n\ge6$  if and only if the Weyl
component $C^\nabla$ of the curvature $R^\nabla$ vanishes. If $M$ is oriented
and $2n=4$,   Condition \eqnref{eq:Appendix} holds if and only if the Weyl
component of the Riemann curvature tensor is self-dual when the signature is
$(4,0)$ or $(0,4)$ and anti-self-dual when the signature is $(2,2)$. 
\end{proposition}

\begin{proof}
Whenever the Weyl tensor vanishes, the remaining two terms $S^\nabla, E^\nabla$
satisfy
\begin{eqnarray*}
\Tr(S_p^\nabla(X,Y)\circ{}j)&=&\frac{scal(g)}{n(2n-1)}g_p(X,jY),\\
\Tr_\bbR \left(E^\nabla_{p} (X, Y)\circ{}j\right)
&=& \frac{1}{n+1}\left(\widehat{Ric}_p(X,jY)-\widehat{Ric}_p(Y,jX)\right)
\end{eqnarray*}
for all $p\in M,\, X,Y \in T_pM, j\in J(M,g)_p,$ and both the right-hand sides
satisfy condition \eqnref{eq:Appendix} as was already mentioned in Proposition
\ref{prop:pseudoRi} . 

The remainder of this section is devoted to the proof of the converse; we use a
construction from the analysis of the curvature in the (positive definite)
almost Hermitian case due to Tricerri and Vanhecke \cite[page
372]{bib:TricerriVanhecke} but which makes sense in the bundle $J(M,g)$ of
compatible almost complex structures where $g$ is pseudo-Riemannian.

Fix $p\in M$, let $V=T_pM$, and $j \in J(M,g)_p$. We set
\[
\scrV^j_3 =  \{S \in \wedge^2V^* \suchthat S(jX,jY) = -S(X,Y)\,\,\forall X,Y \in V \},
\]
then for $S \in \scrV^j_3$ and $\psi_j(S) \in \wedge^2V^*\otimes \End(V)$ defined by
\begin{eqnarray*}
\kern-18pt g_p(\psi_j(S)(X,Y)Z,W) &=& 2g_p(X,jY)S(Z,jW) 
+ 2g_p(Z,jW)S(X,jY) \\
&&\kern-20pt\mbox{}+ g_p(X,jZ)S(Y,jW) + g_p(Y,jW)S(X,jZ) \\
&&\kern-20pt\mbox{}- g_p(X,jW)S(Y,jZ) - g_p(Y,jZ)S(X,jW),
\end{eqnarray*}
$\psi_j(S)$ is in $\CC(V,g_p)$. With $s\in \End V$ defined by $g(sX,Y)=S(X,Y)$,
we have
\begin{eqnarray*}
\psi_j(S)(X,Y)Z&=&-2g_p(X,jY)jsZ - 2S(X,jY)jZ \\
&&- g_p(X,jZ)jsY-S(X,jZ)jY +S(Y,jZ)jX+ g_p(Y,jZ)jsX.
\end{eqnarray*}
A simple computation shows that the Ricci trace of $\psi_j(S)$ is zero
for all $S \in \scrV_3$:
\begin{eqnarray*}
\Tr[Y\mapsto \psi_j(S)(X,Y)Z]&=&2g_p(X,sZ)+ 2S(X,Z)\\
&&\mbox{}- g_p(X,jZ)\Tr(js)-S(X,jZ)\Tr j\\
&&\mbox{}+S(jX,jZ)+g_p(jsX,jZ)\\
&=& 2S(Z,X)+2S(X,Z)-S(X,Z)+S(X,Z)\\
&=&0
\end{eqnarray*}
since $j$ and $js$ are  traceless because $g(s \cdot,\cdot)=S(\cdot,\cdot)$ and
$g(j\cdot,\cdot)$ are skew-symmetric. Hence $\psi_j(S)$ lies in the space of Weyl
tensors. 

\begin{remark}
In \cite{bib:TricerriVanhecke}, where only the positive definite metric case is
discussed, the space $\psi_j(\scrV_3)$ is one of the 10 irreducible components
of the orthogonal Riemann curvature type tensors under the action of the unitary
group and is there called $\scrW_9$. It can be shown to be the only component
with non-vanishing $\Omega^{R,j}_2$. For this reason we make the definition
below in the pseudo-Riemannian case.
\end{remark}

\begin{definition}
Put $\scrW^j_9 = \psi_j(\scrV_3^j)$ then:
\end{definition}

\begin{lemma} \label{lem:W9new}
If $R \in \scrW^j_9$ then $\Omega^{R,j}_2(X,Y) = -8(n+1)S(X,jY)$ where $R=
\psi_j(S)$ with $S \in \scrV_3^j$.
\end{lemma}

\begin {proof}
If $R \in \scrW_9^j$ then $R = \psi_j(S)$ with $S$ an antisymmetric bilinear form in
$\scrV_3^j$, with $S(jX,jY) = -S(X,Y)$ and we have
\begin{eqnarray*}
\Omega^{R,j}_1(X,Y)&=&\Tr\left( \psi_j(S)(X,Y)j\right)\\
&=&-2g_p(X,jY)\Tr (jsj) + 2S(X,jY)\Tr(\Id) \\
&& \qquad\mbox{}+g_p(X,jsY)+g_p(sX,jY)-g_p(sY,jX)-g_p(Y,jsX)\\
&=&4nS(X,jY)-S(Y,jX)+S(X,jY)-S(Y,jX)+S(X,Y) \\
&=&4(n+1)S(X,jY)
\end{eqnarray*}
since $S(Y,jX) = - S(jX,Y) = S(j^2X,jY) = -S(X,jY)$ and also $jsj=s$ so $\Tr(jsj) = 0$.
Then
\[
\Omega^{R,j}_2(X,Y) = 4(n+1)S(jX,j^2Y) - 4(n+1)S(X,jY) =-8(n+1)S(X,jY).
\]
\end{proof}

Let $R \in \CC(V,g_p)$ be any curvature and set $S^{R,j}(X,Y) =
\frac{1}{8(n+1)}\Omega_2^{R,j}(X,jY)$ then Lemma \ref{lem:W9new} implies
$R=\psi_j(S^R)$ when $R \in \scrW_9^j$. We can then define $P_j(R) = \psi_j(S^{R,j})$
for any $R \in \CC(V,g_p)$. 
The following Lemma is obvious.
 
\begin{lemma} \label{lem:obvious}
Let $j \in J(M,g)_p$ and $h \in O(V,g_p)$ Then
\begin{itemize}
\item $P_j$ is a linear endomorphism of the space $\CC(V,g_p)$ of curvature
tensors with ${P_j}^2=P_j$ and with image in $\scrW_9^j$ a subspace of Weyl tensors.
\item $P_{hjh^{-1}} = h P_jh^{-1}$ for the natural action of $O(V,g_p)$ on 
curvature tensors.
\end{itemize}
\end{lemma}

We are now ready to complete the Proof of Proposition \ref{prop:Appendix}. It is
a consequence of Lemma \ref{lem:obvious} that any curvature $R \in \CC(V,g_p)$
with $\Omega_2^{R,j} = 0$ is in the kernel of the projection $P_j$ for each $j
\in J(M,g)_p$ and hence in the intersection of these kernels. This intersection
will then be disjoint from the span $\scrW$ of the images $\scrW^j_9$ of $P_j$
as $j$ varies. From the equivariance property of Lemma \ref{lem:obvious} it
follows that $\scrW$ is a non-zero $O(V,g_p)$-invariant subspace of the Weyl
tensors. But the Weyl tensors are irreducible under the full orthogonal group
when $2n\ge 4$ \cite[page 47]{bib:Besse} so $R$ is of Ricci type. When there is
an orientation, $\scrW$ is a non-zero $SO(V,g_p)$-invariant subspace of the Weyl
tensors. In dimension $2n>4$, the Weyl tensors are irreducible under
$SO(V,g_p)$. In dimension $4$, we compute in a pseudo-orthonormal oriented basis
$\{ e_1,\ldots ,e_4\}$ in which 
$g=\begin{pmatrix}\epsilon_1 &0 &0&0\\0 & \epsilon_2 &0&0\\0 &0&\epsilon_1 &0\\
0&0&0&\epsilon_2 \end{pmatrix}$ and $j=\begin{pmatrix}0 &0 &-1&0\\0 & 0 &0&-1\\1 &0&0 &0\\
0&1&0&0 \end{pmatrix}$; then any $S \in \scrV^j_3 $ has the form 
$S=\begin{pmatrix}0 &A &0&B\\-A & 0 &B&0\\0 &B&0 &-A\\
-B&0&A&0 \end{pmatrix}$. The corresponding Weyl tensor 
$\psi_j(S)(X,Y)$% $=-2g(X,jY)js - 2S(X,jY)j
%+ g(jX,\cdot)jsY-S(jX,\cdot)jY +S(jY,\cdot)jX- g(jY,\cdot)jsX$
 satisfies
\begin{eqnarray*}
\psi_j(S)(e_1,e_2)&=&-\psi_j(S)(e_3,e_4) \\
 \psi_j(S)(e_1,e_3)&=&\epsilon_1 \epsilon_2\psi_j(S)(e_2,e_4)\\
 \psi_j(S)(e_1,e_4)&=&-\psi_j(S)(e_2,e_3).
\end{eqnarray*}
Since the Hodge star dual is given by 
\begin{eqnarray*}
*(e_1\wedge e_2)&=&\epsilon_1 \epsilon_2\,\,\, e_3\wedge e_4\\
*(e_1\wedge e_3)&=&-\, \, e_2\wedge e_4\\
*(e_1\wedge e_4)&=&\epsilon_1 \epsilon_2\,\,\,  e_2\wedge e_3,
\end{eqnarray*}
we see that $\psi_j(S)$ viewed as a map from $\Lambda^2 T^*_pM$ into itself,
vanishes on the $\epsilon_1 \epsilon_2$-eigenspace of the Hodge dual.
This shows that $\scrW$ is the space of anti-self-dual Weyl tensors when 
$\epsilon_1 \epsilon_2=1$
and the space of self-dual Weyl tensors when $\epsilon_1 \epsilon_2=-1$.

This completes
the proof. 
\end{proof}

\section{Symplectic structure} 

We consider  a symplectic manifold $(M,\omega)$ of dimension $2n\ge 4$;  we
shall use in this section the more classical notation of $\omega$ (instead of
$g$) for the symplectic structure. Let $\Omega$ be a non degenerate
skew-symmetric  bilinear form on a real vector space $V$ of dimension $2n$. A
symplectic frame at a point $p$ is a map $\xi: V\rightarrow T_pM$ which is a
linear isomorphism between $(V,\Omega)$ and $(T_pM,\omega_p)$;  as mentioned in
section \ref{section:twistor} the bundle of symplectic frame
$F(M,\omega)\rightarrow M$ is a principal bundle with structure group $G =
Sp(V,\Omega)$ which is isomorphic to the simple split real Lie group $Sp(2n,\bbR)$
when one has chosen a basis of $V$ in which the matrix associated to $\Omega$ is
$\Omega_0=\left(\begin{array}{cc}0&I_n\\-I_n&0\end{array}\right)$.

The twistor bundle $J(M,\omega)\rightarrow M$ has fibre over the point $p$ given
by all complex structures $j$ on $T_pM$ which are compatible with $\omega_p$
(i.e. $\omega_p(jX,jY)=\omega_p(X,Y)$ for all $X,Y\in T_pM$) and positive (i.e.
$\omega_p(X,JX)>0$ for all $0\neq X\in T_pM$).

For the construction of the almost complex structures ${J_{\nabla}^{\pm}}$ on
the twistor bundle $J(M,\omega)$, one chooses a symplectic connection $\nabla$;
this is a linear torsion-free connection such that $\nabla\omega=0$; it is well
known that those exist but are not unique on any symplectic manifold.

\begin{definition}
The space $\CC(T_pM,\omega_p)$ of symplectic curvature type  tensors at a point
$p$ is isomorphic the subspace $\CC(V,\Omega)$ of elements $\wt{R} \in \Wedge^2
V^*\! \otimes \fraksp(V,\Omega)$ satisfying the Bianchi identity 
\[
\CC(V,\Omega) = \biggl\{\wt{R} \in \Wedge^2 V^*\! \otimes \fraksp(V,\Omega) \suchthat 
\cyclic_{X,Y,Z} \wt{R}(X,Y)Z = 0\biggr\}.
\]
where $\fraksp(V,\Omega)$ is the Lie algebra of $Sp(V,\Omega)$ and consists of
endomorphisms $\xi$ of $V$ with $\Omega(\xi X,Y) + \Omega(X, \xi Y) = 0$ for all
$X,Y$ in $V$ or, equivalently, $\Omega(\xi X, Y)$ is a symmetric bilinear form.
\end{definition}

The adjoint representation of $Sp(V,\Omega)$ on $\fraksp(V,\Omega)$ is isomorphic to
the irreducible representation $S^2 V^*$. The following elementary Lemma will be
useful in constructing elements of $\CC(V,\Omega)$.

\begin{lemma}\label{lem:sympcurvsym}
Given an element $A$ of $(\otimes^4 V)^* = \otimes^4 V^*$ satisfying
\begin{enumerate}
\item $A(X,Y,Z,T)$ is anti-symmetric in $X$ and $Y$;
\item $A(X,Y,Z,T)$ is symmetric in $Z$ and $T$;
\item $\cyclic_{X,Y,Z} A(X,Y,Z,T) = 0$
\end{enumerate}
then there is a unique element $\ul{A} \in \CC(V,\Omega)$ such that $A(X,Y,Z,T) =
\Omega(\ul{A}(X,Y)Z,T)$.
\end{lemma}

Given an element $\wt{R} \in \CC(V,\Omega)$ we can form its Ricci trace
$\ric(\wt{R})$ given by
\[
\ric(\wt{R})(X,Y) = \Tr (Z \mapsto \wt{R}(X,Z)Y)
\]
which is a symmetric bilinear form on $V$. This gives a linear map $\ric \colon
\CC(V,\Omega) \longto S^2 V^*$ which is equivariant for the natural actions of
$Sp(V,\Omega)$. Given a symmetric bilinear form $r \in S^2 V^*$ let $\rho^r \in
\fraksp(V,\Omega)$ be defined by
\[
\Omega(\rho^r X,Y) = r(X,Y)
\]
and $E(r)$ by
\begin{eqnarray}
\Omega(E(r)(X,Y)Z,T) &=& \frac{-1}{2(n+1)}\biggl[2\Omega(X,Y)r(Z,T) 
+\Omega(X,Z)r(Y,T) - \Omega(Y,Z)r(X,T) \nonumber\\ 
&&\qquad\mbox{}+r(Y,Z)\Omega(X,T)  - r(X,Z)\Omega(Y,T)\biggr]. \label{eqn:Edef}
\end{eqnarray}

\begin{lemmadef}
$E(r)$ is in $\CC(V,\Omega)$ and $E \colon S^2V^* \longto \CC(V,\Omega)$ is an
equivariant linear map with $\ric(E(r)) = r$. $E(\ric(\wt{R}))$ is called the
Ricci component of $\wt{R}$ and $W(\wt{R}) = \wt{R}-E(\ric(\wt{R}))$ the Weyl
component. If we define 
\[
\EE(V,\Omega) = \{\wt{R} \in \CC(V,\Omega) \suchthat \wt{R}=E(\wt{R})\}
\,\textrm{ and } \,\WW(V,\Omega) = \{\wt{R}\in \CC(V,\Omega) \suchthat
E(\wt{R})=0\}
\] 
then both subspaces are irreducible under the action of $Sp(V,\Omega)$ and 
\[
\CC(V,\Omega) = \EE(V,\Omega) \oplus \WW(V,\Omega).
\]
\end{lemmadef}

\begin{proof} 
To see that \eqnref{eqn:Edef} defines a curvature term we check that the three
properties in Lemma \ref{lem:sympcurvsym} hold which is straight forward. For
the irreducibility see \cite{bib:Vaisman}.
\end{proof}

\begin{definition}
This gives a decomposition of the curvature $R^\nabla$ of a symplectic connection:
\[
R^\nabla=E^\nabla + W^\nabla
\]
where $E^\nabla$ is defined in terms of the Ricci tensor  $Ric^\nabla(X,Y)=\Tr
[Z\rightarrow R^\nabla(X,Z)Y]$; it can be written as
\begin{eqnarray} E^\nabla(X,Y)Z &=& \frac{-1}{2(n+1)}
\biggl[2\omega(X,Y)\rho^\nabla Z +\omega(X,Z)\rho^\nabla Y 
- \omega(Y,Z)\rho^\nabla X \nonumber\\ 
&&\qquad\mbox{}+
Ric^\nabla(Y,Z)X - Ric^\nabla(X,Z)Y \biggr] \label{eqn:Edeff}
\end{eqnarray}
with $\omega(\rho^\nabla X,Y)=Ric^\nabla(X,Y)$ and of course the Weyl component
is $W^\nabla= R^\nabla - E^\nabla$. A symplectic connection $\nabla$ is said to
be of Ricci-type if $W^\nabla=0$, i.e. if $R^\nabla = E^\nabla$.
\end{definition}

Since ${\Omega}_0({\wt{j_0}}X,Y)+{\Omega}_0(X,{\wt{j_0}}Y) =0$, the 4i
eigenvalue can only arise in the $\WW(V,\Omega)$ tensor part and does so, hence
$J^+_\nabla$ is integrable in the symplectic context  if and only if
$W^\nabla=0$, as was observed by Vaisman \cite{bib:Vaisman2}.
 
If the  symplectic connection is of Ricci-type, then
\begin{eqnarray}\label{imageNj}
\textrm{Imaginary part of }\left((\Id-ij)\circ R^\nabla_p\left((\id- ij)X,
(\id - ij) Y \right)\circ (\id+ij)\right)&& \quad\label{eq:imN}\\
&& \kern-10cm=\frac{-2}{n+1}\left[ j, -\ul{X}\otimes B^\nabla_jY -\ul{B^\nabla_jY}\otimes X   
+\ul{Y}\otimes B^\nabla_jX +\ul{B^\nabla_jX}\otimes Y     \right]\nonumber
\end{eqnarray} 
for any $j\in J(M,\omega)_p$, where $B=\rho_p^\nabla -j\rho_p^\nabla j$ and
$\ul{U}=\omega_p(U,\cdot)$, and 
\begin{equation}\label{2formsympl}
\Tr_\bbR (R_p^\nabla(X,Y)\circ j)=-\frac{1}{n+1} 
\left( \omega_p(X,Y)\Tr (\rho_p^\nabla \circ j)+
\omega_p((\rho_p^\nabla \circ j+j\circ \rho_p^\nabla)X,Y)\right).
\end{equation}

\begin{proposition}\label{prop:sympl1}
The almost complex structure $J^+_\nabla$ on the twistor space $J(M,\omega)$ of
a symplectic manifold $(M,\omega)$ of dimension $2n\ge 4$, defined using a
symplectic connection $\nabla$, is integrable if and only if the curvature of
$\nabla$ is of Ricci-type, i.e. $W^\nabla$ vanishes.

The almost complex structure $J^-_\nabla$ is never integrable.

If the  symplectic connection is of Ricci-type, then:
 \begin{itemize}
\item 
the image of the Nijenhuis tensor associated to $J_\nabla^-$ at any point $j\in
J(M,\omega)$, is the whole horizontal tangent space plus the part of the
vertical tangent space given by the endomorphisms defined by formula
\eqnref{eq:imN};
\item 
the closed $2$-form on $J(M,\omega)$ associated by 
\eqnref{def:omegaJ(M,g,nabla)} to $\nabla$, $\omega^{J(M,\omega,\nabla)}$, is
symplectic if and only if
\[
\Tr_\bbR (\rho^\nabla_p \circ j) \Id +  (\rho_p^\nabla \circ j+j\circ \rho_p^\nabla)
\] 
has a vanishing kernel for all $ p\in M$ and all $j\in J(M,\omega)_p$;
\item 
the almost complex structures $J^\pm_\nabla$ are compatible with the symplectic
$2$-form in the sense of equation \eqnref{eq:compat}.
\end{itemize}
\end{proposition}

The remainder of this section is devoted to the study of this compatibility
(equation \eqnref{eq:Appendix}) for a general symplectic connection. We define
(as was done in Definitions \ref{def:Omega1} and \ref{def:Omega2}) for an
element $R \in \CC(V,\Omega)$ and  a $j \in J(V,\Omega)\simeq J(M,\omega)_p$ let
$\Omega^{R,j}_1(X,Y) = \Tr_\bbR \left(R (X, Y)\circ j\right)$ for $X,Y \in V$
and let $\Omega^{R,j}_2(X,Y) = \Omega^{R,j}_1(jX,jY)-\Omega^{R,j}_1(X,Y)$. The
compatibility condition becomes again $\Omega^{R,j}_2=0$ for all  $j$.

\begin{definition}
For $j \in J(V,\Omega)$ we set 
\[
\scrV(V,\Omega,j)= \{ S\in \Lambda^2(V^*) \suchthat S(jX,jY) = -S(X,Y)\}.
\]
\end{definition}

\begin{remark}
As a representation of $U(V,\Omega,j)$, $\scrV(V,\Omega,j)$ is a real
irreducible subspace of $\Lambda^2(V^*)$ and its complexification is
$\Lambda^{(2,0)} \oplus \Lambda^{(0,2)}$.
\end{remark}

\begin{definition}
For $S \in \scrV(V,\Omega,j)$ define $R(S,j)(X,Y)Z \in V$ by
\begin{eqnarray}
\Omega(R(S,j)(X,Y)Z,T) &=& -2\Omega(Z,jT)S(X,jY) 
+ \Omega(X,jZ)S(Y,jT) \nonumber\\
&&\qquad\mbox{}+ \Omega(X,jT)S(Y,jZ) - \Omega(Y,jT)S(X,jZ)\nonumber\\
 &&\qquad\qquad\mbox{}- \Omega(Y,jZ)S(X,jT) \label{eq:symppsi}
\end{eqnarray}
for all $T \in V$.
\end{definition}
The left hand side $\Omega(R(S,j)(X,Y)Z,T)$ is clearly antisymmetric in $X$ and
$Y$, symmetric in $Z$ and $T$ and satisfies the Bianchi identity
$\cyclic_{X,Y,X} R(S,j)(X,Y)Z=0$. A straight forward calculation shows it is
Ricci flat and so in $\WW(V,\Omega)$, moreover we have $\Omega_2^{R(S,j),j}(X,Y)
= -8(n-1)S(X,jY)$. In summary:

\begin{lemma}\label{lem:}
Formula \eqnref{eq:symppsi} defines an element $R(S,j) \in \CC(V,\Omega)$ which
is of Weyl type and $S \mapsto R(S,j)$ is a $U(V,\Omega,j)$ equivariant map
$\scrV(V,\Omega,j) \longto \CC(V,\Omega)$ with image in the Weyl tensors.
Moreover 
\[
S(X,Y) = \frac1{8(n-1)} \Omega_2^{R(S,j),j}(X,jY).
\]
Under the action of $h \in Sp(V,\Omega)$ we have
\[
h \cdot (\scrV(V,\Omega,j)) = \scrV(V, \Omega, hjh^{-1})\quad \mathrm{and}\quad 
h\cdot (R(S,j)) = (h\cdot R)(h\cdot S, hjh^{-1}).
\]
\end{lemma}

\begin{definition}
For arbitrary $R \in \CC(V,\Omega)$ we  define $S^{R,j}(X,Y) =
\frac{1}{8(n-1)}\Omega_2^{R,j}(X,jY)$ and $P_j(R) = R(S^{R,j},j)$.
\end{definition}

\begin{lemma}
$P_j$ is a linear map from $\CC(V,\Omega)$ to itself satisfying $P_j \circ P_j =
P_j$ and with image in the curvatures of Weyl type. $ j \mapsto P_j$ is
$Sp(V,\Omega)$-equivariant.
\end{lemma}

\begin{proposition} \label{prop:sympl}
Let $(M,\omega)$ be a symplectic manifold of dimension $2n \ge4$ with a
symplectic connection $\nabla$. Then the closed $2$-form
$\omega^{J(M,\omega,\nabla)}$ is of type $(1,1)$ for each of the $J^\pm_\nabla$
(i.e. equation \eqnref{eq:Appendix} is satisfied) if and only if the curvature
$R^\nabla$ is of Ricci type.
\end{proposition}

\begin{proof}
If $R^\nabla$ is of Ricci type then $R^\nabla= E^\nabla$ and, as mentioned in
Proposition \ref{prop:sympl1}, a direct calculation involving equation
\eqnref{2formsympl} shows that $\Omega_2^{E^\nabla,j} = 0$ for all $j$.

Conversely, assume $\Omega_2^{R^\nabla,j} = 0$ for all $j$ then as in the
pseudo-Riemannian case this means $R^\nabla$ is in the kernel of $P_j$ for all
$j$ and by equivariance, replacing $j$ by $hjh^{-1}$ it follows that $R^\nabla$
is in the intersection $\cap_h \Ker P_{hjh^{-1}}$ which is a subspace of
$\CC(V,\Omega)$ disjoint from the span of the images of the $P_{hjh^{-1}}$. This
is a non-zero $Sp(V,\Omega)$-invariant subspace of the Weyl curvature tensors
and by irreducibility must be the whole of the Weyl curvatures. Hence $R^\nabla$
has no Weyl curvature so is of Ricci type.
\end{proof}

\end{document}